\documentclass[11pt, reqno, letterpaper]{amsart}

\usepackage{esint}
\usepackage{amssymb}

\usepackage[
hypertexnames=false, colorlinks, 
linkcolor=black,citecolor=black,urlcolor=black, linktocpage=true
]%
{hyperref} 
\hypersetup{bookmarksdepth=3}

	\normalbaselines						  %

\usepackage{graphicx}
\usepackage{
mathrsfs}

\theoremstyle{definition}

\newtheorem*{df*}{Definition}

\theoremstyle{remark}

\newtheorem*{rem*}{Remark}

\numberwithin{equation}{section}

\newtheorem{theorem}{Theorem}

\newtheorem{lemma}{Lemma}
\newtheorem{corollary}{Corollary}

\newtheorem{definition}{Definition}

\newtheorem{remark}{Remark}

\newtheorem*{question}{Question}

\newcommand{\La}{\langle}
\newcommand{\Ra}{\rangle}

\usepackage{graphicx}
\DeclareGraphicsExtensions{.eps}

\newcommand{\s}{\mathbf}

\begin{document}

\title[Bellman functions and Brascamp--Lieb inequality]{Hessian of  Bellman functions and uniqueness of Brascamp--Lieb inequality}

\author{Paata Ivanisvili}
\thanks{PI is partially supported by the Hausdorff Institute for Mathematics, Bonn, Germany}
\address{Department of Mathematics,  Michigan State University, East
Lansing, MI 48824, USA}
\email{ivanisvi@math.msu.edu}

\author{Alexander Volberg}
\thanks{AV is partially supported by the NSF grant DMS-1265549 and by the Hausdorff Institute for Mathematics, Bonn, Germany}
\address{Department of Mathematics,  Michigan State University, East
Lansing, MI 48824, USA}
\email{volberg@math.msu.edu}

\makeatletter
\@namedef{subjclassname@2010}{
  \textup{2010} Mathematics Subject Classification}
\makeatother




%
%


  \begin{abstract}
Under some assumptions on the vectors $a_{1},..,a_{n}
\in\mathbb{R}^{k}$ and the function  $B : \mathbb{R}^{n} \to \mathbb{R}$
we find the  sharp estimate of the expression $\int_{\mathbb{R}^{k}}
B(u_{1}(a_{1}\cdot x),\dots, u_{n}(a_{n}\cdot x))dx$ in terms of
$\int_{\mathbb{R}}u_{j}(y)dy,  j=1,\dots,n.$ In some particular case  we
will show that these assumptions on $B$ imply that there is only one
Brascamp--Lieb inequality.
\end{abstract}

\date{}
\maketitle

\section{Introduction}

\setcounter{equation}{0}
\setcounter{theorem}{0}


\subsection{Brascamp--Lieb inequality}
The classical Young's inequality for convolutions  on the real line asserts that for any  $f \in L^{p}(\mathbb{R})$ and $g \in  L^{q}(\mathbb{R})$ where $p,q \geq 1$,  we have  an inequality 

\begin{align}\label{youngi}
\|f*g\|_{r} \leq  \| f\|_{p}\|g\|_{q}
\end{align}
if and only if 
\begin{align}\label{youngc}
\frac{1}{p}+\frac{1}{q}=1+\frac{1}{r}.
\end{align}
 In what follows $f*g$ denotes convolution i.e. $f*g (y) = \int_{\mathbb{R}}f(x)g(y-x) dx$. 
The necessity of  (\ref{youngc}) follows immediately: by stretching the functions $f$ and $g$  as  $f_{\lambda}(x)=\lambda^{1/p} f(\lambda x) $, $g_{\lambda}(x)=\lambda^{1/q} g(\lambda x) $ corresponding norms do not change. Since $\|f_{\lambda}*g_{\lambda}\|_{r}=\lambda^{\frac{1}{p}+\frac{1}{q}-1-\frac{1}{r}}\|f * g\|_{r}$ we obtain (\ref{youngc}). Beckner (see \cite{beckner}) found the sharp constant $C=C(p,q,r)$ of the inequality  $\|f*g\|_{r} \leq  C \| f\|_{p}\|g\|_{q}$. At the same time (see \cite{brascamp--lieb}) Brascamp and Lieb derived more general inequality, namely, let $a_{1},..,a_{n}$ be the vectors of $\mathbb{R}^{k}$ where $1\leq k \leq n$, let $u_{j} \in L^{1}(\mathbb{R})$ be nonnegative functions,  where $1\leq p_{j} \leq \infty$ and $\sum_{j=1}^{n}\frac{1}{p_{j}}=k$, then we have a sharp inequality 
\begin{align}\label{bl}
\int_{\mathbb{R}^{k}}\prod_{j=1}^{n}u^{1/p_{j}}_{j}(\langle a_{j}, x \rangle) dx \leq D(p_{1},\ldots,p_{n}) \prod_{j=1}^{n}\left( \int_{\mathbb{R}}u_{j}(x)dx \right)^{1/p_{j}},
\end{align}
where $\langle\cdot, \cdot\rangle$ denotes scalar product in Euclidian space, 
 \begin{align}\label{cbl}
 D(p_{1},\ldots, p_{n})  = \sup_{b_{1},..,b_{n} > 0}\int_{\mathbb{R}^{k}}\prod_{j=1}^{n}g^{1/p_{j}}_{j}(\langle a_{j}, x\rangle) dx
 \end{align}
 and  $g_{j}(x) = b_j^{1/2} e^{-\pi x^{2}b_{j}}$. In other words, supremum in (\ref{bl}) is achieved by centered, normalized (i.e. $\|g_{j}\|_{1}=1$) gaussian functions.
Usually inequality (\ref{bl})  is written as follows:
\begin{align*}
\int_{\mathbb{R}^{k}}\prod_{j=1}^{n}w_{j}(\langle a_{j}, x \rangle) dx \leq D(p_{1},\ldots,p_{n}) \prod_{j=1}^{n}\|w_{j}\|_{p_{j}}.
\end{align*}
for $u_{j} \in L^{p_{j}}(\mathbb{R})$. Surely the above inequality becomes the  same as (\ref{bl}) after introducing the functions $w_{j}(x)=u_{j}^{1/p_{j}}(y)$. 

It is clear that Brascamp--Lieb inequality (\ref{bl}) implies sharp Young's inequality for convolutions. Indeed, just take $n=3, k=2, a_{1}=(1,0), a_{2}=(1,-1), a_{3}=(0,1)$ and use duality argument. 

  The next natural question which arose  was the following one: what conditions should  the vectors $a_{j}$ and the numbers $p_{j}$ satisfy in order for the constant $D(p_{1},\ldots,,p_{n})$ to be finite. It turns out that the answer has simple geometrical interpretation which was  first found by Barthe (see \cite{barthe}): we consider all different  $k$-tuples of vectors $(a_{j_{1}},..,a_{j_{k}})$ such that they create basis in $\mathbb{R}^{k}$. All we need from these tuples are the numbers $j_{1},\ldots ,j_{k}$.  Each $k$-tuple defines a unique vector $v \in \mathbb{R}^{n}$ with entries 0 and 1 so that $j_{i}$-th component is 1 ($i=1,\ldots ,k$) and the rests are zeros. Finally we take convex hull of the vectors $v$ and denote it by $K$. The constant $D(p_{1},\dots, p_{n})$ is finite if and only if $\left(\frac{1}{p_{1}},\ldots, \frac{1}{p_{n}}\right) \in K$. In other words, in order to make the set $K$ large we want the vectors $a_{1},\ldots, a_{n}$ to be {\em more linearly independent}. Later the proof of the Brascamp--Lieb inequality (\ref{bl}) was simplified (see \cite{CLM}) by heat flow method. The idea of the method is quite similar to {\em Bellman function technique} which we are going to discuss in the current article. The same idea was used in \cite{BCCT} in order to derive {\em general rank Brascamp--Lieb inequality} (see also \cite{BCCT2}): let $B_{j} : \mathbb{R}^{k}\to \mathbb{R}^{k_{j}}$ be a surjective linear maps,  $u_{j} : \mathbb{R}^{k_{j}} \to \mathbb{R}_{+}$, $k_{j}, k \in \mathbb{N}$, and $p_{j} \geq 1$ are such that $\sum_{j=1}^{n} \frac{k_{j}}{p_{j}}=k$ then we have a sharp inequality 
\begin{align*}
\int_{\mathbb{R}^{k}}u^{1/p_{1}}_{1}(B_{1}x)\cdots u^{1/p_{n}}_{n}(B_{n}x)dx \leq C\left(\int_{\mathbb{R}^{k_{1}}}u_{1}\right)^{1/p_{1}}\cdots \left(\int_{\mathbb{R}^{k_{n}}}u_{n} \right)^{1/p_{n}}
\end{align*}
where 
\begin{align}\label{cbl2}
C = \sup_{A_{1},\ldots, A_{n} >0} \int_{\mathbb{R}^{k}}G^{1/p_{1}}_{1}(B_{1}x)\cdots G^{1/p_{n}}_{n}(B_{n}x)dx
\end{align}
and $G_{j}(y) = e^{-\pi \langle A_{j}y,y\rangle} (\det{A_{j}})^{1/2}$. Supremum in (\ref{cbl2}) is taken over all positive definite $k_{j}\times k_{j}$ matrices $A_{j}$. One of the main result obtained  in \cite{BCCT} describes finiteness of the number $C$, namely, $C$ is finite if and only if 
\begin{align*}
\dim(V) \leq \sum_{j=1}^{n} \frac{\dim{(B_{j}V)}}{p_{j}} \quad \text{for all subspaces} \quad V \subset \mathbb{R}^{k}.
\end{align*}
After this  result the original inequality (\ref{bl}) got a name {\em rank 1 Brascamp--Lieb inequality}. If $k=1$ the  inequality (\ref{bl}) becomes usual multilinear H\"older's inequality 
\begin{align}\label{holder}
\int_{\mathbb{R}} \prod_{j=1}^{n}u^{1/p_{j}}_{j}(x) dx \leq \prod_{j=1}^{n} \left( \int_{\mathbb{R}}u_{j} \right)^{1/p_{j}}  \Longleftrightarrow\;\;\; \sum_{j}\frac{1}{p_{j}}=1.
\end{align}

From the {\em Bellman function point of view}   the multilinear H\"older's inequality holds because the 
following function 
\begin{align}\label{hol}
B(x_{1},\ldots, x_{n})=x_{1}^{1/p_{1}}\cdots x_{n}^{1/p_{n}}
\end{align}
 is concave in the domain $x_{j}\geq 0$ for $\sum_{j=1}^{n}\frac{1}{p_{j}}\leq 1$ (we assume that $p_{j} >0$). 
 
 This Bellman function point of view asks us to look for the description of functions $B$ such that 
 \begin{align}
 \label{postanovka}
\!\!\! \int_{\mathbb{R}^k}B(u_{1}(\langle a_1, x\rangle),\ldots, u_{n}(\La a_n, x\Ra)dx \,\,\text{is estimated in terms of}\,\,\left\{\int_{\mathbb{R}}u_{i}(x)dx\right\}_{i=1}^n\ . 
 \end{align}
 Function $B(x_{1},\ldots, x_{n})=x_{1}^{1/p_{1}}\cdots x_{n}^{1/p_{n}}$, $\sum_{j=1}^{n}\frac{1}{p_{j}}= 1$, is an example of such a function for $k=1$. But for $k=1$ one can easily get the full description of ``Bellman functions" that give inequality \eqref{jensen} below. 
 
 The equality $\sum_{j=1}^{n}\frac{1}{p_{j}}= 1$ was needed because the function $B(x_{1},\ldots,x_{n})$ has to be homogeneous of degree 1 i.e., $B(\lambda \s{x})=\lambda B(\s{x})$. This allows us to write integral over the real line. Indeed, if  the nonnegative  functions $u_{j}$ are integrable then Jensen's inequality implies 
$$
\frac{1}{|I|}\int_{I}B(u_{1},\ldots, u_{n})dx \leq B\left(\frac{1}{|I|}\int_{I}u_{1}dx,\ldots,\frac{1}{|I|}\int_{I}u_{n}dx \right)
$$
where $I$ is any subinterval of the real line. 
Since the function $B$  is 1-homogeneous we can rewrite the above inequality as follows 

$$
\int_{I}B(u_{1},\ldots, u_{n})dx \leq B\left(\int_{I}u_{1}dx,\ldots,\int_{I}u_{n}dx \right).
$$
Take $I=[-R,R]$ and send $R$ to infinity. $B$ is continuos, so that 
$$
B\left(\int_{I}u_{1}(x)dx,\ldots,\int_{I}u_{n}(x)dx \right) \to B\left(\int_{\mathbb{R}}u_{1}dx,\ldots,\int_{\mathbb{R}}u_{n}dx \right)\ .
$$ Continuity of $B$ and monotone convergence theorem implies that 
\begin{align}
\label{jensen}
\int_{\mathbb{R}}B(u_{1}(x),\ldots, u_{n}(x))dx \leq B\left(\int_{\mathbb{R}}u_{1}dx,\ldots,\int_{\mathbb{R}}u_{n}dx \right)
\end{align}

It is worth to formulate the following lemma. Set  $\mathbb{R}^{n}_{+} = \{(x_{1},\ldots,x_{n}) : x_{j} \geq 0\}.$ 
\begin{lemma}\label{jen}
Let $u_{j}$ be nonnegative integrable functions $j=1,\ldots, n$ on the real line. If $B$  is 1-homogeneous concave function on $\mathbb{R}^{n}_{+}$,  then (\ref{jensen}) holds. Equality is achieved in (\ref{jensen}) if $(u_{1}, \ldots, u_{n})$ are all proportional.\end{lemma}
\begin{proof}
As we just saw, the proof  follows from showing that $\int_{I}B(u_{1},\ldots, u_{n}) \to \int_{\mathbb{R}}B(u_{1},\ldots, u_{n}).$ We are going to find now a summable amjorant. Take any point $\s{x}_{0}$ from the interior of $\mathbb{R}^{n}_{+}$. Consider any subgradient $v=(v_{1},\ldots, v_{n})$ at point $\s{x}_{0}$ i.e. $B(\s{x}) \leq \langle v, \s{x}-\s{x}_{0}\rangle +B(\s{x}_{0})$. Take $\s{x} = \lambda \s{x}_{0}$ and use the homogeneity of $B$. Thus we obtain $(\lambda -1 )B(\s{x}_{0})\leq (\lambda -1)\langle v, \s{x}_{0}\rangle$ for any $\lambda \geq 0$. This  means that $B(\s{x}_{0})=\langle v, \s{x}_{0}\rangle$ and, therefore, $B(\s{x}) \leq \langle v, \s{x} \rangle$. On the other hand let $e_{j}=(0,\ldots, 1,\ldots, 0)$ be a basis vectors ($j$-th component entry is $1$ and the rests are zero). Consider any point $\s{x} = (x_{1},.\ldots,x_{n}) \in \mathbb{R}^{n}_{+}$. Concavity  and homogeneity of $B$ implies that $B(x) \geq \sum_{j=1}^{n}x_{j}B(e_{j}).$ So we obtain the majorant 
$$
|B(\s{x})|\leq \max \left\{ \sum_{j=1}^{n} x_{j} |B(e_{j})|,  \sum_{j=1}^{n}x_{j} |v_{j}| \right\} \quad\text{for any}\quad \s{x} \in \mathbb{R}^{n}_{+.}
$$
Plugging $u_j$ for $x_j$ we see that the use of Lebesgue's dominated convergence theorem is justified.
\end{proof}

The above Lemma says that homogeneity and concavity of the function implies the inequality \eqref{jensen}. The converse is also true.

 Now the following question becomes quite natural: 
\begin{question}
Assume $a_{1},\ldots, a_{n} \in \mathbb{R}^{k},$ $B$ is continuous function defined on $\mathbb{R}^{n}_{+}$ and $u_{j}$, $j=1,\ldots, n$ are nonnegative integrable functions. What is the sharp estimate of the expression 
\begin{align}
\label{gbl}
\int_{\mathbb{R}^{k}}B(u_{1}(\langle a_{1}, \s{x}\rangle), \ldots, u_{n}(\langle a_{n}, \s{x}\rangle))d\s{x}
\end{align}
in terms of $\int_{\mathbb{R}}u_{j}$?
\end{question}

In other words, along with Young's functions 
$$
B(x_1,\dots, x_n)= x_1^{1/p_1}\cdot\dots\cdot x_n^{1/p_n}, \,\,\,\,\sum\frac1{p_j} =k,
$$
 what can be other Brascamp--Lieb Bellman functions that would give us sharp estimates of \eqref{gbl}?

We give partial answer on this question. It turns out that if one requires function $B$ is homogeneous of degree $k$ and in addition it satisfies some  mild assumptions (smoothness and exponential integrality given below), then we can find the sharp estimate of the expression (\ref{gbl}) in terms of $\int_{\mathbb{R}}u_{j}$, if $B$ satisfies an interesting concavity condition. In a trivial case $k=1$ our theorem gives us of course inequality (\ref{jensen}).  

In the trivial case $k=1$ we already saw  that the interesting concavity condition mentioned above  is precisely the usual concavity of $B$.
In another trivial case $k=n$, the interesting concavity condition mentioned above becomes ``separate concavity" of $B$ in each of its variables.

For $1<k<n$ our concavity condition is, in fact, some compromise between these two concavities.

\bigskip

As we will see  $k=n-1$ and $k=n$ this concavity condition (plus $k$-homogeneity and mild regularity) imply that Brascamp--Lieb Bellman functions $B$ can be only the standard ones: $B(x_1,\dots, x_n)= x_1^{1/p_1}\cdot\dots\cdot x_n^{1/p_n}, \sum\frac1{p_j} =k$.

\bigskip

Before we start formulating our results, we will explain that there are {\em many} 1-homogeneous concave functions  $B$ on $\mathbb{R}^{n}_{+}$.
\begin{lemma}
Function $B$ is continuous, concave and homogeneous of degree 1 on $\mathbb{R}^{n}_{+}$ if and only if there exists continuous, concave function $\tilde{B}(\s{y})$ on $\mathbb{R}^{n-1}_{+}$ such that $\lim_{\lambda \to \infty} \frac{1}{\lambda} \tilde{B}\left(\lambda \s{y}\right)$ exists, it is continuous with respect to $\s{y}$ and $B(x_{1},\ldots, x_{n}) = x_{1} \tilde{B}\left(\frac{x_{2}}{x_{1}}, \ldots, \frac{x_{n}}{x_{1}} \right)$
\end{lemma}
\begin{proof}
 Indeed, if $B$ is continuous,  concave and homogeneous of degree 1 then  $B(x_{1},\ldots, x_{n}) = x_{1} B\left(1,\frac{x_{2}}{x_{1}}, \ldots, \frac{x_{n}}{x_{1}} \right)$ and the function  $\tilde{B} = B(1, y_{1}, \ldots, y_{n-1})$ is continuous and concave in $\mathbb{R}^{n-1}_{+}$. Moreover,  for each $\s{y} \in \mathbb{R}^{n-1}_{+}$,  $\lim_{\lambda \to \infty} \frac{1}{\lambda} \tilde{B}(\lambda \s{y})$ exists  and it is continuos with respect to $\s{y}$. 
 
Assume $\tilde{B}$ satisfies the conditions of the Lemma.  Consider  $B(x_{1},\ldots, x_{n}) = x_{1} \tilde{B}\left(\frac{x_{2}}{x_{1}}, \ldots, \frac{x_{n}}{x_{1}} \right)$.  It is clear that $B$ is continuous on $\mathbb{R}^{n}_{+}$ and it is homogeneous of degree 1. We will show that $B$ is concave in the interior of $\mathbb{R}^{n}_{+}$ and hence by continuity it will be concave on the closure as well. Let $\s{x}=(x_{1},\ldots, x_{n}), \s{y}=(y_{1},\ldots, y_{n}) \in \mathbb{R}^{n}$. Let $\bar{\s{x}}=(x_{2},\ldots, x_{n})$, $\bar{\s{y}}=(y_{2},\dots,y_{n})$ and $\alpha + \beta =1$  for $\alpha, \beta  \in [0,1]$. Then we have 
\begin{align*}
&B(\alpha \s{x} + \beta \s{y}) = (\alpha x_{1}+\beta y_{1}) B \left(1,\alpha \frac{\bar{\s{x}}}{\alpha x_{1}+\beta y_{1}} + \beta \frac{\bar{\s{y}}}{\alpha x_{1}+\beta y_{1}}\right)\geq (\alpha x_{1}+\beta y_{1}) \times  \\
&\left[\alpha B\left(1, \frac{\bar{\s{x}}}{\alpha x_{1}+\beta y_{1}} \right)+\beta B\left(1, \frac{\bar{\s{y}}}{\alpha x_{1}+\beta y_{1}} \right) \right]=\alpha B(\s{x})+\beta B(\s{y}).
\end{align*}
\end{proof}
\subsection{Bellman function in Brascamp--Lieb inequality}
In what follows we assume that  $B \in C(\mathbb{R}^{n}_{+})\cap C^{2}(\mathrm{int}(\mathbb{R}^{n}_{+}))$. In order for the quantity (\ref{gbl}) to be finite it is necessary to assume that $1\leq k \leq n$. Fix some vectors $a_{j}=(a_{j1},\ldots, a_{jn}) \in \mathbb{R}^{k}$ and $k\times k$ symmetric matrix $C$ such that $\langle Ca_{j}, a_{j}\rangle >0$ for $j=1,\ldots, n$.  Let $A$ be a $k \times n$ matrix  constructed by columns $a_{j}$ i.e. $A=(a_{1},\ldots, a_{n})$. We denote $A^{*}$ transpose matrix of $A$. 

Let $u_{j} : \mathbb{R} \to \mathbb{R}_{+}$ be  such that $0<\int_{\mathbb{R}} u_{j} <\infty$. Let $u_{j}(y,t)$ solves the heat equation $\frac{\partial u_{j}}{\partial t}-\langle Ca_{j},a_{j}\rangle \frac{\partial^{2} u_{j} }{\partial y^{2} }=0$ with the initial value $u_{j}(y,0)=u_{j}(y)$. Let $\mathrm{Hess}\, B(\s{y})$ denotes Hessian matrix of the function $B$ at point $\s{y}$. 

For two square  matrices of the same size $P=\{p_{ij}\}$ and $Q=\{q_{ij}\}$ , let $P\bullet Q=\{p_{ij}q_{ij}\}$ be Hadamard product. Denote by symbol\footnote{Sometimes we will write just $\s{u}$.}  
$$
\s{u}(x,t)=(u_{1}(\langle a_{1},x\rangle,t),\ldots, u_{n}(\langle a_{n},x\rangle,t))
$$
 and   denote 
 $$
 \s{u}'(x,t)=\left(u'_{1}(\langle a_{1},x\rangle,t),\ldots, u'_{n}(\langle a_{n},x\rangle,t)\right),
 $$
  where $u'_{j}(\langle a_{j},x\rangle ,t) = \left.\frac{\partial u_{j}(y,t)}{\partial y}\right|_{y=\langle a_{j},x \rangle}$.

\begin{lemma}
\label{quadratic}
For any $0< t < \infty$  and any $x \in \mathbb{R}^{k}$ we have 
\begin{align}
\label{bellprinc}
\left( \frac{\partial }{\partial t} -\sum_{i,j=1}^{k} c_{ij}\frac{\partial^{2}}{\partial x_{i}\partial x_{j}} \right) B(\s{u}(\s{x},t)) = -\langle (A^{*}CA)\bullet \mathrm{Hess}\,B(\s{u}(\s{x},t))\s{u}'(\s{x},t), \s{u}'(\s{x},t) \rangle.
\end{align}
\end{lemma}
\begin{proof}
First we show that the functions $u_{\ell}(\langle a_{\ell},\s{x}\rangle, t)$ satisfy the following heat equation 
\begin{align*}
\left( \frac{\partial }{\partial t} -\sum_{i,j=1}^{k} c_{ij}\frac{\partial^{2}}{\partial x_{i}\partial x_{j}} \right)u_{\ell}(\langle a_{\ell}, \s{x}\rangle,t)=0,   \quad \text{for any} \quad \ell= 1,\ldots, n.
\end{align*}
Indeed, let $u''_{j}(\langle a_{j},\s{x}\rangle ,t) = \left.\frac{\partial^{2} u_{j}(y,t)}{\partial y^{2}}\right|_{y=\langle a_{j},\s{x} \rangle}$. Then 
\begin{align*}
&\sum_{i,j=1}^{k}c_{ij} \frac{\partial^{2}}{\partial x_{i} \partial x_{j}} u_{\ell}(\langle a_{\ell}, \s{x}\rangle,t)=\sum_{i,j=1}^{k}c_{ij}  \frac{\partial }{\partial x_{i}}\left( a_{\ell j}u'_{\ell}(\langle a_{\ell}, \s{x}\rangle,t) \right)=\\
&\sum_{i,j=1}^{k}c_{ij}   a_{\ell j}a_{\ell i}u''_{\ell}(\langle a_{\ell}, \s{x}\rangle,t)  = \langle Ca_{\ell}, a_{\ell}\rangle  u''_{\ell}(\langle a_{\ell}, \s{x}\rangle,t) = \frac{\partial}{\partial t}u_{\ell}(\langle a_{\ell}, \s{x}\rangle,t).
\end{align*}
Let $\s{u} = \s{u}(\s{x},t)$ and $u_{\ell} = u_{\ell}(\langle a_{\ell},\s{x} \rangle,t)$. Then 
\begin{align*}
&\left( \frac{\partial }{\partial t} -\sum_{i,j=1}^{k} c_{ij}\frac{\partial^{2}}{\partial x_{i}\partial x_{j}} \right) B(\s{u}) = 
\sum_{\ell=1}^{n} \frac{\partial B}{\partial u_{\ell}} \frac{\partial u_{\ell} }{\partial t}-
\sum_{i,j=1}^{k}c_{ij} \frac{\partial }{\partial x_{i}}\left(\sum_{\ell=1}^{n}\frac{\partial B}{\partial u_{\ell}} \frac{\partial u_{\ell}}{\partial x_{j}}\right)=\\
&\sum_{\ell=1}^{n} \frac{\partial B}{\partial u_{\ell}}\left( \sum_{i,j=1}^{k} c_{ij}\frac{\partial^{2}}{\partial x_{i}\partial x_{j}} u_{\ell} \right)-
\sum_{i,j=1}^{k}c_{ij} \left(\sum_{\ell, m=1}^{n}\frac{\partial^{2} B}{\partial u_{m}\partial u_{\ell}} \frac{\partial u_{m}}{\partial x_{i}}\frac{\partial u_{\ell} }{\partial x_{j}} + \sum_{\ell=1}^{n}\frac{\partial B}{\partial u_{\ell}} \frac{\partial^{2} u_{\ell}}{\partial x_{i} \partial x_{j}}\right)=\\
&-\sum_{i,j=1}^{k} \sum_{\ell, m=1}^{n} c_{ij}\frac{\partial^{2} B}{\partial u_{m}\partial u_{\ell}} \frac{\partial u_{m}}{\partial x_{i}}\frac{\partial u_{\ell} }{\partial x_{j}} = -\sum_{i,j=1}^{k} \sum_{\ell, m=1}^{n} c_{ij}\frac{\partial^{2} B}{\partial u_{m}\partial u_{\ell}} a_{mi}a_{\ell j}u'_{m}u'_{\ell}=\\
&-\sum_{\ell, m=1}^{n} \langle Ca_{m}, a_{\ell}\rangle \frac{\partial^{2} B}{\partial u_{m}\partial u_{\ell}} u'_{m}u'_{\ell}=-\langle (A^{*}CA)\bullet \mathrm{Hess}\,B(\s{u})\s{u}', \s{u}' \rangle.
\end{align*}
\end{proof}

\begin{remark}
Let us denote by $\Delta_C:= \sum_{i,j=1}^{k} c_{ij}\frac{\partial^{2}}{\partial x_{i}\partial x_{j}} $. 
We used above that
$$
\left( \frac{\partial }{\partial t} -\Delta_C\right)u_{\ell}(\langle a_{\ell}, \s{x}\rangle,t)=0\ .
$$
 This is exactly the equality that implies
\begin{equation}
\label{dC}
\left( \frac{\partial }{\partial t} -\Delta_C\right)B(u_{\ell}(\langle a_{\ell}, \s{x}\rangle,t))=-\langle (A^{*}CA)\bullet \mathrm{Hess}\,B(\s{u}(\s{x},t))\s{u}'(\s{x},t), \s{u}'(\s{x},t)\rangle\ .
\end{equation}

In other words, we look at the natural ``energy" of the problem $\int B(u_{\ell}(\langle a_{\ell}, \s{x}\rangle,t))\, d\s{x}$ at time $t$, and differentiate it in $t$. Replacing $d/dt$ by $d/dt-\Delta_C$ does not change the result because when we integrate the above equality over $\s{x}$ varibales, we should expect  the term $\int \Delta_C\,B(u_{\ell}(\langle a_{\ell}, \s{x}\rangle,t))\, d\s{x}$ to disappear (and this is exactly what happens below).  But the definite sign in the right hand side of \eqref{bellprinc} guarantees us now the monotonicity property of the energy.

So   composing of the special heat flow $e^{-t\Delta_C}$ and special function $B$ seems like a good idea exactly because of the monotonicity formula, which we are going to obtain shortly below.
\end{remark}

Further we make  several assumptions on the function $B$. The assumption $L3$ is exactly the concavity we were talking about above.
\begin{itemize}
\item[L1.] $B \in C(\mathbb{R}^{n}_{+})\cap C^{2}(\mathrm{int}(\mathbb{R}^{n}_{+}))$.
\item[L2.] $B(\lambda \s{y}) = \lambda^{k}B(\s{y})$ for all $\lambda  \geq 0$ and $\s{y} \in \mathbb{R}^{n}_{+}$.
\item[L3.] There exists $k\times k$ symmetric matrix $C$ such that $(A^{*}CA)\bullet \mathrm{Hess}\,B(\s{y}) \leq 0$ for $\s{y} \in \mathrm{int}(\mathbb{R}^{n}_{+}),$ and $\langle Ca_{j},a_{j} \rangle >0$ for all $j=1,\ldots, n$.
\item[L4.] $B \geq 0$ and $B$ is not identically $0$. 
\item[L5.] 
\begin{align}\label{exp}
\int_{\mathbb{R}^{k}}B(e^{-\langle a_{1},\s{x} \rangle^{2}}, \ldots, e^{-\langle a_{n},\s{x} \rangle^{2}}) d\s{x} <\infty.
\end{align}
\end{itemize}

We make several observations: properties $L3$ and $L4$ imply that the function $B$ is  separately concave (i.e. concave  with respect to each variable) and increasing with respect to each variable, moreover, $B> 0$ in $\mathrm{int}(\mathbb{R}^{n}_{+})$. The above properties imply that 
\begin{align}\label{exp2}
\int_{\mathbb{R}^{k}} B(b_{1}e^{-\delta_{1}\langle a_{1},\s{x} \rangle^{2}}, \ldots,b_{n} e^{-\delta_{2}\langle a_{n},\s{x} \rangle^{2}}) d\s{x} < \infty
\end{align}
for any positive numbers $b_{j}, \delta_{j} >0$.
  
Consider the following class of functions $E(\mathbb{R})$: $u \in E(\mathbb{R})$ if and only if  there exist constants $b, \delta >0$ such that $|u(y)| \leq b e^{-\delta y^{2}}$. It is clear that if $u\in E(\mathbb{R})$ then 
$u(y,t) \in E(\mathbb{R})$ for any $t \geq 0$ where $u(y,t)$ denotes heat extension of $u(y)$ i.e. $u(y,t)=u(y,0)$ and $\frac{\partial }{\partial t} u(y,t)  =\sigma \frac{\partial^{2}}{\partial y^{2}} u(y,t)$ with some $\sigma >0$. Note that $E(\mathbb{R})$ contains the functions with compact support. Also note that  if nonnegative functions $u_{j}$ belong to the class $E(\mathbb{R})$ then the following function 
\begin{align}\label{bell}
\s{B}(t) = \int_{\mathbb{R}^{k}}B(u_{1}(\langle a_{1},x  \rangle,t), \ldots, u_{n}(\langle a_{n},x  \rangle,t)) d\s{x}. 
\end{align} 
is finite for any $t \geq 0$.

\begin{lemma}
\label{l:vanish}
Let $u_{j}$ be nonnegative functions from $E(\mathbb{R})$. 
Then for any $t \in (0,\infty)$ we have 

\begin{align*}
\lim_{r\to \infty}\int_{V_{r}} \sum_{i,j=1}^{k} c_{ij}\frac{\partial^{2}}{\partial x_{i} \partial x_{j}} B(\s{u}(\s{x},t)) dx =0 \quad  \text{where}\quad  V_{r}=\{ x \in \mathbb{R}^{k} : \|x\|\leq r\}.
\end{align*}
\end{lemma}

\begin{proof}
Let $F(x) = B(e^{-\langle a_{1},x\rangle^{2}}, \ldots, e^{-\langle a_{n},x\rangle^{2}})$. Let $x=r \sigma$ where $\sigma \in \mathbb{S}^{k-1}_{1}$.  Since $B$ is increasing with respect to each components,  for each $\sigma$ the function $F(r\sigma)$ is decreasing with respect to $r$. Therefore the function $\tilde{F}(r)=\int_{\mathbb{S}^{k-1}_{1}} F(r\sigma)d\sigma_{1}$ is decreasing. Here $\sigma_{r}$ denotes surface measure of the sphere $\mathbb{S}^{k-1}_{r}$ or radius $r$. Since $\int_{0}^{\infty} \tilde{F}(r) r^{k-1}dr<\infty$ we obtain 
\begin{align*}
R^{k}\tilde{F}(2R)\leq R \min_{R\leq r\leq 2R} \tilde{F}(r)r^{k-1}\leq \int_{R}^{2R}\tilde{F}(r)r^{k-1}dr.
\end{align*}
This implies that $\lim_{r \to \infty} r^{k}\tilde{F}(r)=0$. 

By Stokes' formula we have 
\begin{align*}
\int_{V_{r}} \frac{\partial^{2}}{\partial x_{i} \partial x_{j}} B(\s{u}(\s{x},t)) = 
\int_{\partial V_{r}} \frac{\partial}{ \partial x_{j}} B(\s{u}(\s{x},t)) n_{i} d\sigma_{r}
\end{align*}
where $n_{i}$ is the $i$-th component of the unit normal vector to the boundary of the ball $V_{r}$. 

Homogeneity of $B$ implies that $\sum _{j=1}^{n}\frac{\partial}{\partial y_{j}} B(\s{y}) y_{j}=kB(\s{y}).$ Since $\frac{\partial}{\partial y_{j}} B(\s{y}) \geq 0$ we obtain  estimate $\frac{\partial }{\partial y_{j}} B(\s{y})y_{j} \leq kB(\s{y})$. Also we note that for each $t>0$ there exists a constant $L$ depending on the parameters $t, u_{j}$ such that $\big|\frac{\partial }{\partial y} u_{j}(y,t)\big|\leq L y u_{j}(y,t)$ for all $y \in \mathbb{R}$. 

So we obtain that 
\begin{align*}
&\left| \int_{\partial V_{r}} \frac{\partial}{ \partial x_{j}} B(\s{u}(\s{x},t)) n_{i} d\sigma_{r} \right| \leq  \sum_{\ell=1}^{n} \int_{\partial V_{r}} \frac{\partial B}{\partial u_{\ell}} \left| \frac{\partial u_{\ell} (\langle a_{\ell},\s{x} \rangle, t )}{\partial x_{j}} \right| d\sigma_{r}\leq \\
&c\int_{\partial V_{r}} r B(\s{u}(\s{x},t))d \sigma_{r} \leq C_{1}r^{k}\tilde{F}(C_{2}r)
\end{align*}
where constants $C_{1},C_{2}$ do not depend on $r$. Since $B$ is homogeneous and it is increasing with respect to each components the last inequality follows from the the observation $B(\s{u}(\s{x},t)) \leq C_{3} F(C_{2} x)$ where $C_{3}, C_{2}$ do not depend on $x$.   So the lemma follows. 

\end{proof}

\begin{remark}\label{compact}
Lemma \ref{l:vanish} holds even if we take supremum with respect to $t$ over any compact subset of $(0, \infty)$.
\end{remark}

\begin{corollary}
\label{increasing}
The function $\s{B}(t)$ is increasing  for $t >0$, and it is constant if and only if  $(A^{*}CA)\bullet \mathrm{Hess}\,B(\s{u}(\s{x},t))\s{u}'(\s{x},t)=\s{0}$ for all $x \in\mathbb{R}^{n}$ and any $t >0$. 
\end{corollary}
\begin{proof}
First we integrate (\ref{bellprinc}) with respect to $t$ (over any closed interval $[t_{1},t_{2}] \subset (0,\infty)$) and then we integrate other the balls $V_{r}$.  Thus the corollary is immediate consequence of Lemmas \ref{quadratic}, \ref{l:vanish} and Remark \ref{compact}.
\end{proof}

Thus we obtain an inequality $\s{B}(t_{1})\leq \s{B}(t_{2})$  for $0<t_{1} \leq t_{2} <\infty$ and we want to pass to the limits.
\begin{lemma}
\label{limits}
Let $B$ satisfies assumptions $L1-L5$ and let $u_{j} \in E(\mathbb{R})$ be nonnegative (not identically zero) functions. Then the following  equalities hold
\begin{align}
&\lim_{t \to 0}\,\s{B}(t)=\int_{\mathbb{R}^{k}}B(u_{1}(\langle a_{1}, \s{x}\rangle), \ldots, u_{n}(\langle a_{n}, \s{x}\rangle))d\s{x}, \label{lim1}\\
&\lim_{t \to \infty}\,\s{B}(t)=\int_{\mathbb{R}^{k}}B\left( \frac{e^{-\frac{ \langle a_{1},\s{x} \rangle^{2} }{\langle Ca_{1},a_{1} \rangle }}}{\sqrt{ \pi\langle Ca_{1},a_{1}\rangle}}\int_{\mathbb{R}}u_{1}dx, \ldots,  \frac{e^{-\frac{\langle a_{n},\s{x} \rangle^{2} }{\langle Ca_{n},a_{n} \rangle }}}{\sqrt{\pi \langle Ca_{n},a_{n}\rangle}}\int_{\mathbb{R}}u_{n}dx\right)d\s{x}. \label{lim2}
\end{align}
\end{lemma}
\begin{proof}
Take any nonnegative (not identically zero) functions $u_{j} \in E(\mathbb{R})$. Then there exist positive numbers $\beta_{j}, \delta_{j}$ such that  $u_{j}(y) \leq \beta_{j} e^{-\delta_{j}y^{2}}$ for all $j=1,\ldots, n$.  Note that 
\begin{align*}
&u_{j}(y,t) = \frac{1}{(4\pi t \langle C a_{j},a_{j} \rangle )^{1/2}}\int_{\mathbb{R}}u_{j}(x) e^{-\frac{(y-x)^{2}}{4t \langle C a_{j},a_{j} \rangle}}dx \leq \\
& \frac{\beta_{j}}{\sqrt{1+4t\delta_{j} \langle Ca_{j},a_{j} \rangle}} e^{-\frac{y^{2}\delta_{j}}{1+4t\delta_{j} \langle Ca_{j},a_{j} \rangle}}
\end{align*}
So the first limit (\ref{lim1}) follows immediately from Lebesgue's dominated convergence theorem. For the second limit (\ref{lim2}) we use homogeneity of the function $B$. So by changing variable $\s{x} = \s{y}\sqrt{t}$ we obtain
\begin{align*}
&\int_{\mathbb{R}^{k}}B\left( \ldots ,
\frac{1}{(4\pi t \langle C a_{j},a_{j} \rangle )^{1/2}}\int_{\mathbb{R}}u_{j}(x) e^{-\frac{(\langle a_{j},\s{x} \rangle -x)^{2}}{4t \langle C a_{j},a_{j} \rangle}}dx,\ldots\right)d\s{x}=\\
&\int_{\mathbb{R}^{k}}B\left( \ldots ,
\frac{e^{-\frac{\langle a_{j},\s{y} \rangle^{2}}{4 \langle C a_{j},a_{j} \rangle}}}{(4\pi  \langle C a_{j},a_{j} \rangle )^{1/2}}\int_{\mathbb{R}}u_{j}(x) e^{\frac{2\sqrt{t}\langle a_{j},\s{y} \rangle x - x^{2}}{4t \langle C a_{j},a_{j} \rangle}}dx,\ldots\right)d\s{y}.
\end{align*}
It is clear that that for each fixed $\s{y}$ integrand tends to 
$$
B\left( \ldots ,
\frac{e^{-\frac{\langle a_{j},\s{y} \rangle^{2}}{4 \langle C a_{j},a_{j} \rangle}}}{(4\pi  \langle C a_{j},a_{j} \rangle )^{1/2}}\int_{\mathbb{R}}u_{j}(x) dx,\ldots\right).
$$
Since $u_{j}(x) \leq b_{j} e^{-\delta_{j} x^{2}}$ we obtain 
\begin{align*}
u_j(x)e^{\frac{2\sqrt{t}\langle a_{j},\s{y} \rangle x }{4t \langle C a_{j},a_{j} \rangle}}
\le b_j e^{-\frac{\delta_j}2 x^2} e^{\max_{x\ge 0}[-\frac{\delta_j}2 x^2 +\alpha_j(t) x]},
\end{align*}
where $\alpha_j(t):= \frac{\La a_j, \s{y}\Ra}{2\sqrt{t} \La C a_j, a_j\Ra}$.
Hence
\begin{align*}
u_j(x)e^{\frac{2\sqrt{t}\langle a_{j},\s{y} \rangle x }{4t \langle C a_{j},a_{j} \rangle}}
\le b_j e^{-\frac{\delta_j}2 x^2} e^{\frac12\frac{\alpha_j(t)^2}{\delta_j}}= b_j e^{-\frac{\delta_j}2 x^2} e^{\frac{\La a_j, \s{y}\Ra^2}{8t\delta_j \La C a_j, a_j\Ra^2}},
\end{align*}
Now we can apply Lebesgue's dominated convergence theorem twice.

The last display estimate gives us a summable majorant for the integration in $x$. On the other hand,
$$
e^{-\frac{\langle a_{j},\s{y} \rangle^{2}}{4 \langle C a_{j},a_{j} \rangle}}e^{\frac{\La a_j, \s{y}\Ra^2}{8t\delta_j \La C a_j, a_j\Ra^2}} \le e^{-\frac{\langle a_{j},\s{y} \rangle^{2}}{8 \langle C a_{j},a_{j} \rangle}}
$$
for all $t\ge t_C$. Thus we get the uniform in $t$ estimate  for the $j$th argument of function $B$: 
$$
\frac{\gamma_j\,e^{-\frac{\langle a_{j},\s{y} \rangle^{2}}{8 \langle C a_{j},a_{j} \rangle}}}{(4\pi  \langle C a_{j},a_{j} \rangle )^{1/2}}\,,
$$
where $\gamma_j:= \int_{\mathbb{R}}b_{j}e^{-\frac{\delta_j}2 x^2}dx$.
Therefore we have the summable majorant (that it is summable follows from $L3$)
$$
B\left( \ldots ,
\frac{\gamma_j\,e^{-\frac{\langle a_{j},\s{y} \rangle^{2}}{8 \langle C a_{j},a_{j} \rangle}}}{(4\pi  \langle C a_{j},a_{j} \rangle )^{1/2}},\ldots \right)\,,
$$
and the lemma is proved.
\end{proof}

Corollary \ref{increasing} and Lemma \ref{limits} imply the following theorem.
\begin{theorem}\label{mainth}
Let $B$ satisfies assumptions $L1-L5$ and let $u_{j} \in E(\mathbb{R})$ be nonnegative (not identically zero) functions. Then we have 
\begin{align}
&\int_{\mathbb{R}^{k}}B(u_{1}(\langle a_{1}, \s{x}\rangle), \ldots, u_{n}(\langle a_{n}, \s{x}\rangle))d\s{x} \leq \label{mainin}\\
&\int_{\mathbb{R}^{k}}B\left( \frac{e^{-\frac{ \langle a_{1},\s{x} \rangle^{2} }{\langle Ca_{1},a_{1} \rangle }}}{\sqrt{ \pi\langle Ca_{1},a_{1}\rangle}}\int_{\mathbb{R}}u_{1}, \ldots,  \frac{e^{-\frac{\langle a_{n},\s{x} \rangle^{2} }{\langle Ca_{n},a_{n} \rangle }}}{\sqrt{\pi \langle Ca_{n},a_{n}\rangle}}\int_{\mathbb{R}}u_{n}\right)d\s{x}.\nonumber
\end{align}

Equality holds if and only if 
\begin{align}\label{eqcase}
(A^{*}CA)\bullet \mathrm{Hess}\,B(\s{u}(\s{x},t))\s{u}'(\s{x},t)=\s{0} \quad \text{for all} \quad \s{x} \in\mathbb{R}^{n}\quad  \text{and any} \quad  t >0.
\end{align}
\end{theorem}

\begin{remark}
So any function satisfying our strange concavity condition $L3$, homogeneity condition $L2$ and some mild conditions $L1, L4,L5$ gives a certain Brascamp--Lieb inequality. Our next goal will be to show that in interesting cases the finiteness of (\ref{exp}) implies that there is basically only one such $B$.
\end{remark}

In the {\em Bellman function technique} theorems of the above type  are known as a first part of the Bellman function method which is usually simple. Any function $B$ that satisfies properties $L1-L5$ will be  called Bellman function of Brascamp--Lieb type. 

The difficult technical part is how to find such Bellman functions. It is worth mentioning that the property $L3$ in principle requires solving partial differential inequalities. We are going to give partial answer on this question in the following section. 

\section{How to find the Bellman function}
\label{findB}
\begin{definition}
Let $\s{y} = (y_{1},y_{2},\ldots, y_{n})  \in \mathrm{int}(\mathbb{R}^{n}_{+})$ and let $D(\s{y})$ be a diagonal  square matrix such that on the diagonal it has the terms $\frac{y_{j}}{\langle Ca_{j},a_{j}\rangle}, j=1,\ldots, n$. 
\end{definition}

\begin{theorem}\label{t2}
If the function $B$ satisfies assumptions $L1-L5$ then we have 
\begin{align}\label{osn-ner}
AD(\s{y})[A^{*}CA\bullet \mathrm{Hess}\,B(\s{y})]=\s{0} \quad \text{for all} \quad \s{y} \in \mathrm{int}(\mathbb{R}^{n}_{+}).
\end{align}
\end{theorem}

\begin{remark}
Equality \eqref{osn-ner} is a second order partial differential equation on $B$. However, assumptions $L1-L5$ are either of quantitative nature, or in the form of partial differential {\it inequalities}.
So it is quite surprising that based only on  assumptions $L1-L5$ one can expect equality  \eqref{osn-ner}. 
\end{remark}

The proof of the above equality is interesting in itself. 

\begin{proof}
We saw in the previous section that assumptions $L1-L5$ imply the inequality (\ref{mainin}).  One can easily observe that the following functions 
\begin{align*}
u_{j}(y)=b_{j} \frac{e^{-\frac{y^{2}}{\langle Ca_{j},a_{j} \rangle }}}{\sqrt{\pi \langle C a_{j},a_{j} \rangle }}, \quad b_{j} >0.
\end{align*}
give equality in the inequality (\ref{mainin}). Since $\s{u}'(\s{x},t)=\frac{-2}{4t+1}D(\s{u}(\s{x},t))A^{*}\s{x}$,  Theorem \ref{mainth} implies that 
\begin{align*}
A^{*}CA\bullet \mathrm{Hess}\,B(\s{u}(\s{x},t)) D(\s{u}(\s{x},t))A^{*}\s{x} =\s{0}
\end{align*} 
Choose any $\s{x} \in \mathbb{R}^{k}$, any $\s{y} \in \mathrm{int}(\mathbb{R}^{n}_{+})$ and any $t>0$. We can find $b_{1},\ldots, b_{n} >0$ such that 
\begin{align*}
u_{j}(\langle a_{j},\s{x}\rangle ,t )=b_{j} \frac{e^{-\frac{\langle a_{j},\s{x} \rangle ^{2}}{\langle Ca_{j},a_{j} \rangle  (4t+1)}}}{\sqrt{\pi \langle C a_{j},a_{j} \rangle (4t+1) }}=y_{j}, \quad j=1,\ldots,n.
\end{align*}
Hence we obtain 
\begin{align*}
[A^{*}CA\bullet \mathrm{Hess}\,B(\s{y})]D(\s{y})A^{*}\s{x}=0, \quad \forall \s{x} \in \mathbb{R}^{k}, \quad \forall \s{y} \in \mathrm{int}(\mathbb{R}^{n}_{+}).
\end{align*}
So equality \eqref{osn-ner} follows. 
\end{proof}

Theorem \ref{t2} implies the following corollary. 
\begin{corollary}\label{rankest}
For any $\s{y} \in \mathrm{int}(\mathbb{R}^{n}_{+})$ we have 
\begin{align}\label{rankin}
\mathrm{rank}(A^{*}CA\bullet \mathrm{Hess}\,B(\s{y}))\leq n-k.
\end{align}
\end{corollary}

The above corollary immediately follows from the fact that $\mathrm{rank}(AD(\s{y}))=\mathrm{rank}(A)=k$ and, for example, from the Sylvester's rank inequality. 

\medskip

Thus for each fixed $n$ we  have a range  of admissible dimensions $1 \leq  k \leq n$. For  the boundary cases $k=1$ and $k=n$, we find the Bellman function with the properties $L1-L5$.  For the intermediate cases $1<k<n$ we partially find the function $B$. 

\subsection{ Case $k=1$.  Jointly concave and homogeneous function} We want to see that in this case $L1-L5$ gives us precisely convex and $1$-homogeneous functions.  In the case $k=1$ we have $A=(a_{1},\ldots, a_{n}) \in \mathbb{R}^{n}$. Since the condition $\langle Ca_{j},a_{j}\rangle >0$ must hold, the $1\times 1$ matrix $C$ must be a positive number and $a_{j} \neq 0$ for all $j=1,\ldots, n$. The fact that $B$ is homogeneous of degree $1$ and $B$ is increasing with respect to each variable immediately imply $L5$. The only property we left to ensure is $L3$. For $\s{v}=(v_{1},\ldots, v_{n}) \in \mathbb{R}^{n}$ let $d(\s{v})$ denotes $n \times n$  diagonal matrix with entries  $v_{j}$ on the diagonal. 
\begin{align*}
A^{*}CA\bullet \mathrm{Hess}\,B(\s{y})=C\cdot A^{*}A\bullet \mathrm{Hess}\,B(\s{y})=C\cdot d(A)\mathrm{Hess}\,B(\s{y}) d(A)\ .
\end{align*}
So the inequality $A^{*}CA\bullet \mathrm{Hess}\,B(\s{y})\leq 0$ is equivalent to the inequality $\mathrm{Hess}\,B(\s{y})\leq 0$, because $C$ is just a number. Thus we obtain the following lemma.

\begin{lemma}
If the function $B$ satisfies assumptions $L1-L5$ then $a_{j} \neq 0$ for all $j$, $C$ is any positive number and $B \in C(\mathbb{R}^{n}_{+})\cap C^{2}(\mathrm{int}(\mathbb{R}^{n}_{+}))$ is a concave homogeneous function of degree 1. Conversely, if $a_{j}\neq 0$ for all $j$ and $B \in C(\mathbb{R}^{n}_{+})\cap C^{2}(\mathrm{int}(\mathbb{R}^{n}_{+}))$ is a nonnegative, not identically zero,  concave, homogeneous function of degree 1 then $B$ satisfies assumptions $L1-L5$. 
\end{lemma}
The above lemma gives complete characterization of the Bellman function in the case $k=1$, and the inequality (\ref{mainin}) is the same as inequality (\ref{jensen}) (see Lemma \ref{jen}).

\subsection{Case $k=n$. $B(\s{y})=Const\cdot \,y_{1}\cdots y_{n}$}
We  show that in the case $k=n$ the assumptions $L1-L5$ are satisfied if and only if $B(\s{y})=M y_{1}\cdots y_{n}$ where $M$ is a positive number. We present 3 different proofs (according to their chronological order), each of them uses different assumptions on $B$ in necessity part. Sufficiency follows immediately. Indeed, if $B=My_{1}\cdots y_{n}$ then all the assumptions $L1-L5$ are satisfied except that one has to check existence of the symmetric matrix $C$ such that $A^{*}CA\bullet \mathrm{Hess}\, B \leq 0$ and $\langle Ca_{j},a_{j} \rangle >0$. But it is enough to take $C=(AA^{*})^{-1}$. Now we go to proving necessity.

\subsubsection{First proof}
As we already mentioned the assumptions $L1-L5$ imply that $B \in C^{2}(\mathrm{int}(\mathbb{R}^{n}_{+}))\cap C(\mathbb{R}^{n}_{+})$ is nonnegative,  separately concave, and it is  homogeneous of degree $n$. We need to show that such $B$ then must have the form  $B(\s{y})=My_{1}\cdots y_{n}$. To show this,  we consider a function $G$ such that  $G(\ln z_{1},\ldots, \ln z_{n}) = \frac{B(z_{1},\ldots, z_{n})}{z_{1}\cdots z_{n}}$ for $z_{j} >0$. Homogeneity of order $0$ of  $B$ implies that $\mathrm{div }\, G =0$, and concavity of $B$ with respect to each variable implies that $\frac{\partial G}{\partial y_{j}}+\frac{\partial^{2} G}{\partial y_{j}^{2}}\leq 0$ for $j=1,\ldots, n$. After summation of the last inequalities we obtain that $G$ is superharmonic function on $\mathbb{R}^{n}$.
But then it is easy to check that if $\Delta G\le 0$, then $\Delta g \ge 0$, where $g:= e^{-G}$. We get a bounded subharmonic function $g, 0\le g\le 1,$ in the whole space. It is well known that then $g$ must be constant. This implies  implies that $G$ is a constant. 

\subsubsection{Second proof}
The second proof immediately follows from the following lemma which does not use any assumptions regarding smoothness of $B$. 
\begin{lemma}
If a function $B$ defined on $\mathbb{R}^{n}_{+}$ is nonnegative on the boundary of $\mathbb{R}^{n}_{+}$, and it is  separately concave and homogeneous of degree $n$ then $B(y_{1},\ldots, y_{n})=My_{1}\cdots y_{n}$ for some real number $M$. 
\end{lemma}
\begin{proof}
The idea is almost as follows: we are going to construct superharmonic function in the bounded domain such that it is nonnegative on the boundary and it achieves zero value at an interior point of the domain. This implies that the constructed function is identically zero.  

 Consider a function $G(y_{1},\ldots, y_{n})=B(\s{y})-B(1,\ldots,1)y_{1}\cdots y_{n}$. Take any cube $Q=[0,R]^{n}$ where $R>1$. The function $G$ is separately concave and it is zero on the diagonal of the cube $Q$ i.e. $G(y,\ldots, y)=0$ for $y \in [0,R]$. $G$ is nonnegative on the whole boundary of the cube $Q$. Indeed, $G$ is zero at the point $(R,\ldots, R)$ and it is nonnegative at point $(R,\ldots, R,0)$, so separate concavity implies that $G$ is nonnegative on the set $(R,\ldots,R,t)$ where $t \in [0,R]$. Similar reasoning implies that $G$ is nonnegative on the whole boundary of the cube $Q$. 

Suppose now $G$ is not zero at some interior point of the cube $Q$, say at point $W$. Take any interior point $A_{0}$ of the cube $Q$ such that $G(A_{0})
=0$. Take a sequence of points $A_{1},...,A_{n}$ belonging to the interior of $Q$ such that the segments $A_{j}A_{j+1}$ (the segment in $\mathbb{R}^{n}$ with the endpoints $A_{j}, A_{j+1}$) are collinear to one of the vector $e_{k}=(0,\ldots, 1,\ldots, 0)$ (on the $k$-th position we have $1$ and the rest of the components are zero) for all $j=0,..,n-1$, and the same is true for the segment $A_{n}W$. Then clearly $G$ is zero on the segment $A_{0}A_{1}$. Indeed, It is zero at point $A_{0}$. Take  a line joining the points $A_{0},A_{1}$. This line intersects the boundary of the cube $Q$,  and $G$ is concave on the line. Since $G$ is nonnegative at the points of the intersection and it is zero at point $A_{0}$ we obtain that $G$ is zero on the part of the line which lies in the cube $Q$. In particular, it is zero at $A_1$. By induction we obtain that $G$ is zero at the points $A_{2},..,A_{n},W$. So the lemma follows. 

\end{proof}
\subsubsection{Third proof}
In this proof let us assume that $B$ is infinitely differentiable in $\int(\mathbb{R}^{n}_{+})$. The assumptions $L1-L5$ imply that $B$ must be a  separately concave. Therefore for the assumption $L3$ we can choose $C=(AA^{*})^{-1}$. Then (\ref{rankin}) implies that $\frac{\partial^{2} B}{\partial y_{j}^{2}}=0$ for all $j=1,\ldots, n$. We claim that if  $B$  satisfies the system of differential equations $\frac{\partial^{2} B}{\partial y_{j}^{2}}=0$ for all $j=1,\ldots, n$ then it has a form  
\begin{align}\label{linearform}
c_{0}+\sum_{k=1}^{n} \left(\sum_{i_{p}\neq i_{q}, \; i_{1},\ldots,i_{k}=1}^{n}c_{i_{1}\ldots i_{k}}\prod_{j=1}^{k}y_{i_{j}} \right)
\end{align}
where the second summation is taken over the pairwise different indexes. Indeed, proof is by induction over the dimension $n$. If $n=1$ the claim is trivial. Since $\frac{\partial^{2} B}{\partial y_{1}^{2}}=0$ we have $B(\s{y})=y_{1}B_{1}(y_{2},\ldots, y_{n})+B_{2}(y_{2},\ldots,y_{n})$. The condition $\frac{\partial^{2} B_{1}}{\partial y_{j}^{2}}=\frac{\partial^{3}}{\partial y_{1} \partial^{2} y_{j}}B=0$ for $j=2,\ldots, n$ implies that $B_{1}$ satisfies hypothesis of the claim. On the other hand,
$B_2 = B(0, y_{2},\dots, y_n)$, and so $B_2$ has less variables, but satisfies the same system of differential equations.


Homogeneity of $B$ implies that $B(\s{y})=c y_{1}\cdots y_{n}$.

\begin{remark}
The second proof is a modification of the proof shown to us by Bernd Kirchheim, we express our gratitude to him.
\end{remark}

\subsection{Case $k=n-1$. Young's function.}
\begin{theorem}\label{youngth}
If $B$ satisfies assumptions $L1-L5$ and $B_{y_{i}y_{j}}\neq 0$ in $\mathrm{int}(\mathbb{R}^{n}_{+})$ for all $i, j =1,\ldots, n$ then $B(\s{y})=My_{1}^{\alpha_{1}}\cdots y_{n}^{\alpha_{n}}$ for some $M>0$ and $0<\alpha_{j}<1$, $j=1,\ldots,n$ 
\end{theorem}
In the end of the section we present  F.~Nazarov's examples which show that if we remove the condition $B_{x_{i}x_{j}}\neq 0$ in the  Theorem \ref{youngth} then the conclusion of the theorem does not hold. It is also worth mentioning that in the classical case when $n=3$ and $k=2$ we obtain that under the assumptions $L1-L5$ and $B_{y_{i}y_{j}}\neq 0$ there are {\em only} Young's inequalities  for convolution. 

\begin{proof}
Equality (\ref{osn-ner}) is the same as 
\begin{align}\label{osn-ner-simple}
\sum_{j=1}^{n} y_{j}a_{js}B_{y_{\ell} y_{j}} \frac{\langle Ca_{\ell},a_{j}\rangle}{\langle Ca_{j},a_{j}\rangle}=0, \quad \forall \ell=1,\ldots, n, \quad \forall s=1,\ldots, k.
\end{align}

We introduce a vector function $P(\s{x})=(p^{1}(\s{x}),\ldots,p^{n}(\s{x}))$, where $\s{x} \in \mathbb{R}^{n}$, such that 
$$
P(\ln y_{1},\ldots, \ln y_{n})=\nabla B(\s{y}).
$$ 
Then equality (\ref{osn-ner-simple}), the fact that $B_{y_{i}y_{j}}=B_{y_{j}y_{i}}$ and homogeneity of $B$ combined imply the following 
\begin{align}
&\langle \nabla p^{\ell}, w_{\ell s}\rangle =0, \quad \forall \ell=1,\ldots, n, \quad \forall s=1,\ldots, k; \label{grad}\\
&e^{-x_{j}}p^{i}_{x_{j}}=e^{-x_{i}}p^{j}_{x_{i}}, \quad i,j=1,\ldots, n; \label{smooth}\\
&\mathrm{div} \,p^{\ell}=(k-1)p^{\ell}, \quad \ell=1,\ldots,n. \label{homo}
\end{align}
where 
\begin{align}\label{vectors}
w_{\ell s} = \left(a_{1s}\frac{\langle Ca_{\ell},a_{1}\rangle}{\langle Ca_{1},a_{1}\rangle}, \ldots, a_{ns}\frac{\langle Ca_{\ell},a_{n}\rangle}{\langle Ca_{n},a_{n}\rangle}\right), \quad \forall \ell=1,\ldots, n, \quad \forall s=1,\ldots, k. 
\end{align}
Now we show that the assumptions $B_{y_{i}y_{j}} \neq 0$ imply that $\langle Ca_{i},a_{j}\rangle \neq 0$ for all $i,j=1,\ldots, n$.

Indeed, suppose that $\langle Ca_{i_{0}},a_{j_{0}}\rangle =0$ for some $i_{0}$ and $j_{0}$. Assumption $L3$ implies that $i_{0} \neq j_{0}$. Since $C$ is symmetric we get that $\langle Ca_{j_{0}},a_{i_{0}}\rangle =0$. Corollary \ref{rankest}  says now that $\mathrm{rank}(A^{*}CA\bullet \mathrm{Hess}\,B(\s{y}))\leq 1$. So the determinant of any $2\times2$ submatrix of $A^{*}CA\bullet \mathrm{Hess}\,B=\left\{ \langle Ca_{i},a_{j}\rangle B_{y_{i}y_{j}}\right\}_{ij}$  (2-minor) is zero. Consider $2\times 2$ submatrix of $A^{*}CA\bullet \mathrm{Hess}\,B$ with the following entries: $(i_{0},i_{0}),(i_{0},j_{0}), (j_{0},i_{0})$ and $(j_{0},j_{0})$. Since its determinant is zero and we assumed $\langle Ca_{i_{0}},a_{j_{0}}\rangle =0$, we get that $\langle C a_{i_{0}}, a_{i_{0}}\rangle \langle C a_{j_{0}}, a_{j_{0}}\rangle =0$. This contradicts to our assumption $L3$. 

Thus we obtain that for each fixed $\ell$ the vectors $w_{\ell s}$, $s=1,\ldots, n,$ span $k=n-1$ dimensional subspace $W_{\ell}$. Therefore, equality (\ref{grad}) implies that $\nabla p^{\ell}(\s{x})=\lambda^{\ell}(\s{x}) v^{\ell}$ where $\lambda^{\ell}(\s{x})$ is a nonvanishing  scalar valued function in $\mathrm{int}(\mathbb{R}^{n}_{+})$,  $v^{\ell} \perp W_{\ell}$ and none of the components of $v^{\ell}$ is zero. 

The  equality (\ref{homo}) implies that we can choose $v^{\ell}$ so that $\langle v^{\ell}, \s{1}\rangle=k-1$ (here $\s{1}=(1,\ldots, 1)\in \mathbb{R}^{n}$) and so that  $\lambda^{\ell}(\s{x})=p^{\ell}(\s{x})$ for all $\ell=1,\ldots, n.$

Hence the equation $\nabla p^{\ell}(\s{x})=p^{\ell}(\s{x})v^{\ell}$ easily implies that 
$p^{\ell}(\s{x})=e^{\langle v^{\ell}, \s{x} \rangle }p^{\ell}(\s{0})$ for all $\ell$. The equalities (\ref{smooth}) imply that 
\begin{align*}
v^{\ell}=(q_{1}, \ldots, q_{\ell-1}, q_{\ell}-1, q_{\ell+1},\ldots, q_{n}), \quad \forall \ell=1,\ldots. n.
\end{align*}
where $\s{q}=(q_{1},\ldots, q_{n}) \in \mathbb{R}^{n}$. It also follows that $P(\s{0})=k\s{q}$ for some number $k\neq 0$. Thus we get that $B_{y_{\ell}} = kq_{\ell} y_{1}^{q_{1}}\cdots y_{n}^{q_{n}}/y_{\ell}$ and this proves Theorem \ref{youngth}.  
\end{proof}

\subsubsection{Example of necessity of the assumption $B_{y_{i}y_{j}}\neq 0$ in Theorem \ref{youngth}}
 
Let $n=3,\, k=2$ and $B(x_{1},x_{2},x_{3})=\varphi(x_{1},x_{2})x_{3}$ where $\varphi \in C^{2}(\mathrm{int}\,\mathbb{R}^{2}_{+})\cap C(\mathbb{R}^{2}_{+})$ is an arbitrary concave function and  homogeneous of degree 1. Let 
\begin{align*}
&A=   \left( {\begin{array}{ccc}
             0 & 0 &1/\sqrt{2}   \\
             1  & 1  &  0             
                \end{array} } \right),            
&C=
            \left( {\begin{array}{cc}
             2  & 0  \\
             0  & 1               
                \end{array} } \right). 
\end{align*}
Then 
\begin{align*}
A^{*}CA=
            \left( {\begin{array}{ccc}
             1  & 1  &  0  \\
             1  & 1  &  0  \\
             0  & 0  &  1               
                \end{array} } \right).
\end{align*}
Since  $\varphi$ satisfies homogeneous Monge--Amp\`ere equation we have $A^{*}CA\bullet \mathrm{Hess}\,B \leq 0$. Clearly all the assumptions $L1-L5$ are satisfied.

\subsubsection{Theorem \ref{youngth} does not hold in the case $1<k<n-1$}
It turns out that even if $B_{y_{i}y_{j}}\neq 0$ and $1<k<n-1$ then it is not necessarily true that a function $B$ which satisfies assumptions $L1-L5$ has a form $B=My_{1}^{\alpha_{1}}\cdots y_{n}^{\alpha_{n}}$. This means that Theorem \ref{youngth} cannot be improved. We give an example in a general case. 

Assume that $1<k<n-1$ and $n >3$  (case $n=3$ was already discussed above). Take arbitrary nonnegative  $\varphi \in C^{2}(\mathrm{int}(\mathbb{R}^{2}_{+}))\cap C(\mathbb{R}^{2}_{+}))$ so that $\varphi$ is a concave function and homogeneous of degree one. 
We choose $\varphi$ so that it has nonzero second derivatives. We consider the following function 
\begin{align}
\label{n-2}
B(\s{y})=y_{1}^{\alpha_{1}}\cdots y_{n-2}^{\alpha_{n-2}}\cdot \varphi(y_{n-1},y_{n}), \quad \s{y} \in \mathbb{R}^{n}_{+}.
\end{align}
Let $a_{n-1}=a_{n}=(0,\ldots, 0,1) \in \mathbb{R}^{k}$ and let $a_{1}=(\tilde{a}_{1},0), \ldots, a_{n-2}=(\tilde{a}_{n-2},0)\in\mathbb{R}^{k}$. 
We choose vectors $\tilde{a}_{1}, \ldots, \tilde{a}_{n-2} \in \mathbb{R}^{k-1}$ in the following way. First of all they span $\mathbb{R}^{k-1}$.
Intersection of the interior of the  convex hull $K$ (described in the Introduction and constructed by the vectors $\tilde{a}_{1}, \ldots, \tilde{a}_{n-2}$) with the hyperplane $\{y_{1}+y_{2}+\cdots +y_{n-2}=k-1\}$ is nonempty.  We choose a point $(\alpha_{1}, \ldots, \alpha_{n-2})$ from this intersection. 

Then there exists  $(k-1)\times (k-1)$ symmetric, positive semidefinite matrix $\tilde{C}$  such that $\tilde{A}^{*}\tilde{C}\tilde{A} \bullet \mathrm{Hess}\,\tilde{B}  \leq 0$ where $\tilde{A}=(\tilde{a}_{1},\ldots, \tilde{a}_{n-2})$ and 
$$
\tilde{B}(y_{1}, \ldots, y_{n-2})=y_{1}^{\alpha_{1}}\cdots y_{n-2}^{\alpha_{n-2}}\ .
$$
Moreover, we have $\langle \tilde{C}\tilde{a}_{j}, \tilde{a}_{j}\rangle >0$ for $j=1,\ldots, n-2$. The existence of such a matrix  $\tilde C$ follows from the solution of the Euler--Lagrange equation for the right side of \eqref{mainin} (see \cite{CLM}, Theorem~5.2), see also Subsection \ref{Young} below. It is clear that the function $B$ satisfies all properties $L1-L5$ except one has to check the property $L3$.  We choose $C$ as follows 
\begin{align*}
C=
            \left( {\begin{array}{cc}
             \tilde{C}  & \s{0}  \\
             \s{0}^{T}  & 1               
                \end{array} } \right). 
\end{align*}

Function $B$ from \eqref{n-2} satisfies $L3$ (and of course it can easily be made to satisfy all other properties $L1-L5$), but it is not a Young function.



\subsection{Case of Young's function}
\label{Young}
In this subsection we consider classical case when $B(\s{y})=y_{1}^{1/p_{1}}\cdots y_{n}^{1/p_{n}}$ where $1\leq p_{j} < \infty$. Assumptions $1\leq p_{j}$ follows from the assumption $L3$ (which implies in particular that the function $B$ is separately concave) and the assumption $p_{j}<\infty$ was made because otherwise we have a function of less variables $m<n$. Note that we also must require that $\sum \frac{1}{p_{j}}=k$. This function satisfies all assumptions of $L1-L5$ except of $L3$. We try to understand for which matrix $A$ and numbers $p_{j}$ there is a matrix $C$ mentioned in the assumption $L3$. The answer on this question was obtained in \cite{CLM}  by using Euler--Lagrange equation. 

 We will obtain equation on the matrix $C$.  

 Note that $\mathrm{Hess}\,B=B\cdot\left \{\frac{1}{p_{i}p_{j}y_{i}y_{j}}\right\}-B\cdot\left \{\frac{\delta_{ij}}{p_{i}y_{i}^{2}}\right\}$ where $\delta_{ij}=1$ if $i=j$, and otherwise it is zero. Therefore equality (\ref{osn-ner}) becomes 
\begin{align*}
A\left\{\frac{y_{i}}{\langle Ca_{i},a_{i}\rangle}\right\}\left[A^{*}CA\bullet \left(B\cdot\left \{\frac{1}{p_{i}p_{j}y_{i}y_{j}}\right\}-B\cdot\left \{\frac{\delta_{ij}}{p_{i}y_{i}^{2}}\right\} \right) \right]=0
\end{align*}
After simplification we obtain
\begin{align}
\label{findC}
A\left\{\frac{1}{p_{i}\langle Ca_{i},a_{i}\rangle}\right\}A^{*}C = I_{k\times k}
\end{align}

Notice that the rank of $A\left\{\frac{1}{\sqrt{p_{i}\langle Ca_{i},a_{i}\rangle}}\right\}$ is $k$ because the rank of $A$ is $k$. Then $k\times k$ matrix $A\left\{\frac{1}{p_{i}\langle Ca_{i},a_{i}\rangle}\right\}A^{*}$is invertible by  Binet--Cauchy formula.
Then we can find $C$ from \eqref{findC} by the following obvious formula
\begin{equation}
\label{C}
C= \left(A\, \,diag\left\{\frac{1}{p_{i}\langle Ca_{i},a_{i}\rangle}\right\}A^{*}\right)^{-1}
\end{equation}
if we can solve the following system of non-linear equations fefining $\langle Ca_j, a_j\rangle$, $j=1,\dots, n$:

\begin{equation}
\label{system}
\langle C a_j, a_j\rangle = \langle \left(A\, \,diag\left\{\frac{1}{p_{i}\langle Ca_{i},a_{i}\rangle}\right\}A^{*}\right)^{-1} a_j, a_j\rangle\ .
\end{equation}

Using the notations 
$$
s_j^2:= \frac1{p_j\langle C a_j, a_j\rangle},\,\, j=1, \dots, n,
$$
we readily transfer \eqref{system} to
\begin{equation}
\label{systems}
\frac1{p_j} = s_j^2\langle \left(A\, \,diag\left\{s_j^2\right\}A^{*}\right)^{-1} a_j, a_j\rangle,\,\, j=1, \dots, n,
\end{equation}
which is precisely equation (3.12) of \cite{CLM}. In  \cite{barthe}, \cite{CLM} it is proved that for $\{\frac1{p_j}\}_{j=1}^n$ in the interior of the convex set $K$ from \cite{barthe}, \cite{CLM} this system \eqref{systems} has a solution. In particular, $C$ as in \eqref{C} does exist.

Notice also, that the Young's functions found by Brascamp--Lieb \cite{brascamp--lieb} and corresponding to the interior of the convex set $K$ from \cite{barthe}, \cite{CLM}, do satisfy all properties $L1-L5$. Only $L3$ is interesting because we need to show that there exists a certain matrix $C$. We just found a certain $C$ in \eqref{C} (when the system \eqref{system} has a solution).  This matrix $C$ will satisfy $L3$ when $B$ is the Young's function $B(\s{y})=y_{1}^{1/p_{1}}\cdots y_{n}^{1/p_{n}}$ where $1<p_{j} < \infty, \sum_{j=1}^n 1/ p_{j}  =k$. In fact, $A^*CA\bullet Hess B(\s{y})\le 0$ for such a $B$ is equivalent to

$$
diag\left\{\frac1{y_jp_j}\right\} A^* CA\, \,diag\left\{\frac1{y_jp_j}\right\} \le diag\left\{\frac{\langle Ca_j, a_j\rangle p_j}{y_j^2p_j^2}\right\}.
$$
This is immediately equivalent to
$$
A^* CA \le diag\{1/s_j^2\}.
$$
But denoting $S=diag \{s_j\}$ we make this inequlity $(AS)^* C (AS)\le I_{n\times n}$, which holds because $(AS)^* C (AS)$ is an orthogonal projection onto the span of the columns of $S(A^*)$.

So, we repeat, that the Young's functions found by Brascamp--Lieb \cite{brascamp--lieb} and corresponding to the interior of the convex set $K$ from \cite{barthe}, \cite{CLM}, do satisfy all properties $L1-L5$.
But it is more interesting that, as we have shown above, in certain situations all functions satisfying $L1-L5$ must be of the form of a Young function found by Brascamp and Lieb.

\end{document}